\documentclass[11pt,reqno]{amsart}
\usepackage{amsfonts,amssymb,amsmath,amsthm,adjustbox,tikz-cd}
\usepackage[margin=1.025in]{geometry}
\usepackage[hidelinks]{hyperref}
\usepackage{enumerate}
\usepackage{microtype}
\usepackage{mathtools}
\usepackage{cite}
\usepackage{color}
\usepackage{tikz-cd}
\usepackage{float}

\def\set#1{\left\{ #1 \right\}}

\def\Z{\mathbb{Z}}
\def\Q{\mathbb{Q}}

\def\F{\mathbb{F}}
\def\P{\mathbb{P}}
\def\Gal{\operatorname{Gal}}
\def\GL{\operatorname{GL}}

\def\det{\operatorname{det}}

\def\im{\operatorname{im}}

\def\Aut{\operatorname{Aut}}

\theoremstyle{plain}
\newtheorem{theorem}{Theorem}

\newtheorem{lemma}[theorem]{Lemma}

\theoremstyle{definition}

\newtheorem{example}[theorem]{Example}

\theoremstyle{remark}

\title{A uniform bound on the smallest surjective prime of an elliptic curve}
\date{\today}
\author{Tyler Genao}
\address{Tyler Genao, Department of Mathematics, The Ohio State University, Columbus, OH 43210}
\email{genao.5@osu.edu}
\author{Jacob Mayle}
\address{Jacob Mayle, Department of Mathematics, Wake Forest University, Winston-Salem, NC 27109}
\email{maylej@wfu.edu}
\author{Jeremy Rouse}
\address{Jeremy Rouse, Department of Mathematics, Wake Forest University, Winston-Salem, NC 27109}
\email{rouseja@wfu.edu}

\begin{document}

\begin{abstract} Let $E/\mathbb{Q}$ be an elliptic curve without complex multiplication. A well-known theorem of Serre asserts that the $\ell$-adic Galois representation $\rho_{E,\ell^\infty}$ is surjective for all but finitely many prime numbers $\ell$. Considerable work has gone into bounding the largest possible nonsurjective prime; a uniform bound of $37$ has been proposed but is yet unproven. We consider an opposing direction, proving that the smallest prime $\ell$ such that $\rho_{E,\ell^\infty}$ is surjective is at most $7$. Moreover, we completely classify all elliptic curves $E/\mathbb{Q}$ for which the smallest surjective prime is exactly $7$.
\end{abstract}

\maketitle

\section{Introduction}

Let $E$ be an elliptic curve defined over $\Q$. For a positive integer $N$ and a prime number $\ell$, we consider the mod $N$ and $\ell$-adic Galois representations of $E$, which are the homomorphisms
\begin{align*}
\rho_{E,N} &\colon \Gal(\overline{\Q}/\Q) \longrightarrow \Aut(E[N]) \overset{\sim}{\longrightarrow} \GL_2(\Z/N\Z) \\
\rho_{E,\ell^\infty} &\colon \Gal(\overline{\Q}/\Q) \longrightarrow \Aut(T_\ell(E)) \overset{\sim}{\longrightarrow} \GL_2(\Z_\ell)
\end{align*}
that describe the action of the absolute Galois group  $\Gal(\overline{\Q}/\Q)$ on the $N$-torsion subgroup $E[N]$ and the $\ell$-adic Tate module $T_\ell(E) := \varprojlim E[\ell^k]$ of $E$, respectively. 

In the case where $E$ has complex multiplication (henceforth CM), the additional geometric endomorphisms of $E$ impose constraints on the Galois image, preventing $\rho_{E,\ell^\infty}$ from being surjective for all primes $\ell$. However, when $E$ has no complex multiplication (henceforth non-CM), a foundational result of Serre, known as ``Serre's open image theorem,'' gives the following.

\begin{theorem}[Serre \cite{MR387283}, 1972]
    If $E/\mathbb{Q}$ is a non-CM elliptic curve, then $\rho_{E,\ell^\infty}$ is surjective for all sufficiently large prime numbers $\ell$.
\end{theorem}

Serre also asked whether there is a {uniform} bound, \emph{independent of $E$,} on the largest prime number $\ell$ for which $\rho_{E,\ell^\infty}$ can be nonsurjective. There are two known $j$-invariants for which elliptic curves are nonsurjective at $\ell=37$: these are $j=-7 \cdot 11^{3}$ and
$j = -7 \cdot 137^{3} \cdot 2083^{3}$. These both arise from points on the modular curve $X_{0}(37)$, whose rational points were determined in the 1970s \cite{MazurSwD, Velu}. 

Throughout this paper, we say that a prime $\ell$ is \emph{surjective} (resp.\ \emph{nonsurjective}) for $E$ if $\rho_{E,\ell^\infty}$ is surjective (resp.\ nonsurjective). In view of the continued theoretical progress and extensive numerical data that has become available in recent years, it is now widely conjectured \cite{3482279, Zy2015a} that $37$ is the largest possible nonsurjective prime. That is, for all non-CM elliptic curves $E/\Q$ and for all primes $
\ell$,
\begin{equation} \label{E:SUQ}
\ell > 37 \quad \implies \quad \rho_{E,\ell^\infty} \text{ is surjective}.
\end{equation}
Considerable effort has gone toward proving that \eqref{E:SUQ} holds, and while it remains an open problem, significant progress has been made. By studying rational points on modular curves, a series of results \cite{MR387283,MR0644559,Ma1978,MR2753610}, most recently including a work of Furio and Lombardo \cite{furio2023serresuniformityquestionproper}, shows that 
\[
\ell > 37 \quad \implies \quad \rho_{E,\ell^\infty} \text{ is surjective} \quad \textit{or} \quad \im \rho_{E,\ell} = C_{\rm{ns}}^+(\ell)
\]
where $C_{\rm{ns}}^+(\ell) \subseteq \GL_2(\F_\ell)$ denotes the normalizer of a nonsplit Cartan subgroup of $\GL_2(\F_\ell)$.

A related difficult question is the problem of bounding the number of nonsurjective primes. The LMFDB \cite{LMFDB} contains a collection of elliptic curves together with images of $\rho_{E,\ell^{\infty}}$. Examining this data suggests that there is no non-CM elliptic curve $E/\Q$ with four or more nonsurjective primes, and the only
cases of three nonsurjective primes that arise are $\{2,3,5\}$ and $\{2,3,7\}$. In \cite[Theorem 8.1(2)]{BELOV}, it is shown that one way a non-cuspidal, non-CM isolated point on $X_{1}(n)$ can arise is from an elliptic curve with three nonsurjective primes $3 < \ell_{1} < \ell_{2}$. Motivated by this data, we consider the problem of giving a uniform bound on the smallest surjective prime for non-CM elliptic curves.

\begin{theorem} \label{T:main} If $E/\mathbb{Q}$ is a non-CM elliptic curve, then the smallest surjective prime for $E$ is at most $7$. Moreover, if the $j$-invariant of $E$ is not one of the six $j$-invariants in the set
\[  \set{
-2^{-3}\cdot 5^2\cdot 241^3,
-2^4\cdot 3^2\cdot 13^3,
-2^{-5}\cdot 5\cdot 29^3,
-2^{-1}\cdot 5^2,
2^4\cdot 3^3,
2^{-15}\cdot 5\cdot 211^3},
\]
then the smallest surjective prime for $E$ is at most $5$.
\end{theorem}

An immediate consequence of Theorem \ref{T:main} is that at least one of $\rho_{E,2}$, $\rho_{E,3},$ or $\rho_{E,5}$ must be surjective for any non-CM elliptic curve $E/ \mathbb{Q}$. This follows since the mod $2$ Galois representation of an elliptic curve with one of the six $j$-invariants listed in the theorem is necessarily surjective. Furthermore, our result is sharp in a rather strong sense. As a consequence of the proof of \cite[Theorem A]{MR3957898}, there are infinitely many $j$-invariants of elliptic curves for which the smallest surjective prime is $5$. On the other hand, an elliptic curve with one of the six $j$-invariants listed in Theorem \ref{T:main} has $7$ as its smallest surjective prime. We now give an example of such an elliptic curve.

\begin{example}
    Consider the elliptic curve $E/\mathbb{Q}$ given by the Weierstrass equation
    \[     y^2+xy+y=x^3-126x-552.    \]
    This curve has LMFDB \cite{LMFDB} label \texttt{50.a1}, is non-CM, and has $j$-invariant
    \[
    j_E = -2^{-3}\cdot 5^2\cdot 241^3.
    \]
    One can check (and it is noted on the curve's LMFDB page) that
    \[
    \im \rho_{E,2^\infty} = \texttt{8.2.0.1}, \quad \im \rho_{E,3^\infty} = \texttt{3.8.0.2}, \quad \text{and} \quad \im \rho_{E,5^\infty} =    \texttt{5.24.0.4},
    \]
    where the subgroups are identified using their Rouse--Sutherland--Zureick-Brown (RSZB) labels \cite{MR4468989}. Furthermore, $\rho_{E,\ell^\infty}$ is surjective for all primes $\ell \geq 7$. Thus, $7$ is the smallest surjective prime for $E$.
\end{example}

The question of bounding the smallest surjective prime for elliptic curves was previously considered in an article of Larson and Vaintrob \cite[Lemma 22, Theorem 25]{MR3161774}, which noted that for non-CM elliptic curves over a number field $K$, the smallest surjective prime is bounded above by a constant depending only on $K$. However, their bound is ineffective, even in the case $K = \mathbb{Q}$, due to the ineffectiveness of Faltings's theorem. They also proved an effective bound that depends on the conductor of the elliptic curve. This question has also arisen in the context of Jacobians of genus $2$ curves over $\mathbb{Q}$. The article \cite{MR4732686} builds on work of Dieulefait \cite{MR1969642} to implement a practical algorithm for computing the nonsurjective primes associated with the Jacobian $A$ of a genus $2$ curve with a trivial geometric endomorphism ring. To bound the complexity of the algorithm, the authors (one of whom is the second-named author) proved in \cite[Proof of Corollary 3.16]{MR4732686} a bound on the smallest surjective prime of $A$, in terms of the conductor of $A$.

In the present article, we restrict our attention to elliptic curves over $\mathbb{Q}$. Here, we note that an affirmative answer to Serre's uniformity question would immediately yield a uniform (though non-optimal) bound on the smallest surjective prime. If it could be proven that the genus $8$ modular curve $X_{\rm{ns}}^+(19)$ has no non-cuspidal, non-CM rational points, then this would similarly provide a uniform (though still non-optimal) bound on the smallest surjective prime.

The main tool we use is the theory of modular curves $X_H$ associated with subgroups $H \subseteq \GL_{2}(\Z/N\Z)$, which we discuss in Section \ref{S:Preliminaries}. The proof of Theorem~\ref{T:main} is given in Section \ref{S:Approach}. It boils down to studying finitely many modular curves $X_H$ and determining their rational points. One observation that simplifies our work is the enumeration
of ``curious groups'' by Daniels--Gonz\'alez-Jim\'enez \cite{DG}, and more systematically by
Chiloyan \cite{chiloyan2023classificationcuriousgaloisgroups}. In particular, there are two examples (both given in \cite{DG}) of a group $H$ for which $X_{H}$ is a positive rank elliptic curve and an index $3$ subgroup $H_{2} \subseteq H$ for which $X_{H_{2}}$ is also a positive rank elliptic curve, where the natural map $X_{H_{2}} \to X_{H}$ is a bijection on rational points. Our work relies in a crucial way on computations, which we carry out in Magma \cite{MAGMA}. The code  accompanying this paper is available through a public repository \cite{GMR-GitHub}: \\

\centerline{\url{https://github.com/maylejacobj/SmallestSurjectivePrime}}

\subsection{Acknowledgments} This article stemmed from some initial conversations at the Palmetto Number Theory Series (PANTS) XXXVIII conference at Wake Forest University. We are thankful to the organizers (one of whom was the third-named author) and funders of this conference.

\section{Preliminaries} \label{S:Preliminaries}

We begin with a brief discussion of modular curves, since they will serve as our main tool. For a more thorough introduction, we refer the reader to \cite{Zy2015a,zywina2022explicit,MR337993}. Modular curves parameterize elliptic curves with a specified level structure. More specifically, given a subgroup $H \subseteq \GL_{2}(\Z/N\Z)$ that contains $-I \coloneqq \begin{bsmallmatrix} -1 & 0 \\ 0 & -1 \end{bsmallmatrix}$, there is a curve $X_{H}$ together with a map $j \colon X_{H} \to \P^{1}$ characterized by the property that, for any number field $K$, an
elliptic curve $E/K$ with $j(E) \not\in \{ 0, 1728 \}$ has the image of $\rho_{E,N} \colon \Gal(\overline{\Q}/K) \to \GL_2(\Z/N\Z)$ contained in $H$ if and only if $j(E)$ is contained in the set $j(X_{H}(K))$ (see \cite[Proposition 3.3]{Zy2015a}). The preimages of $(1 : 0)$ under $j \colon X_{H} \to \mathbb{P}^{1}$ are called \emph{cusps} of $X_H$.

The curve $X_{H}$ is smooth and projective, and if the determinant map $\det\colon H \to (\Z/N\Z)^{\times}$ is surjective, then $X_{H}$ is geometrically integral and defined over $\Q$. If the genus of $X_{H}$ is at least $2$, then by Faltings's theorem,  $X_{H}(\Q)$ is finite. Under the assumption that $\det \colon H \to (\Z/N\Z)^{\times}$ is surjective, the genus of $X_{H}$ tends to grow with the index of $H$ in $\GL_{2}(\Z/N\Z)$. For this reason, we will seek subgroups $H$ of sufficiently large index so that $X_{H}(\Q)$ is finite. 

Provably finding all rational points on a curve of genus at least $2$ is challenging in general,
but in this paper, we rely only on three relatively simple techniques. First, if $X/\Q$ is a curve with $X(\Q_{p}) = \emptyset$ for some prime $p$, then $X(\Q) = \emptyset$. Second, if $X$ has genus $2$ and the rank of the Jacobian of $X$ is $0$, then the rational points can be effectively determined using Chabauty's method, which is implemented in Magma in this case. Finally, if $X$ admits a morphism $\phi \colon X \to E$ to an elliptic curve $E/\Q$ that is known to have rank $0$, then we can effectively determine $E(\Q)$, and the rational points on $X$ are given by $X(\Q) = \phi^{-1}(E(\Q))$.

We will use the Rouse--Sutherland--Zureick-Brown (RSZB) label \texttt{N.i.g.n} to identify subgroups $H \subseteq \GL_2(\Z/N\Z)$. Here, $N$ denotes the level of $H$, $i$ denotes the index of $H$ in $\GL_{2}(\Z/N\Z)$, $g$ denotes the genus of the modular curve $X_H$, and $n$ is a tiebreaker as described in \cite[Section 2.4]{MR4468989}. With some abuse of notation, we also refer to $X_H$ by the RSZB label of $H$. For example, we will write ${\texttt{4.2.0.1}}$ to  denote both the subgroup of $\GL_2(\Z/4\Z)$ and the modular curve $X_{\texttt{4.2.0.1}}$.  Several modular curves we will encounter have established names in the literature, and in these cases, we will use their common names instead. For example, $X_0(2)$ denotes the modular curve associated with the group \texttt{2.3.0.1}.

We will be particularly interested in modular curves that are themselves fiber products of modular curves. If $N_1$ and $N_2$ are relatively prime and $H_{1}$ and $H_{2}$ are subgroups of $\GL_{2}(\Z/N_1\Z)$ and $\GL_{2}(\Z/N_2\Z)$, respectively, then the modular curve associated with the group
\[
  H \coloneqq \{ M \in \GL_{2}(\Z/N_1 N_2\Z) : M \bmod{N_1} \in H_{1} \text{ and } M \bmod{N_2} \in H_{2} \}
\]
is the fiber product of $X_{H_{1}}$ and $X_{H_{2}}$ over the $j$-line, which we denote by $X_{H_{1}} \times X_{H_{2}}$. This curve fits into the following commutative diagram:

\begin{center}
\begin{tikzcd}
X_{H_{1}} \times X_{H_{2}} \arrow[r] \arrow[d] & X_{H_{1}} \arrow[d, "j_{H_{1}}"] \\
X_{H_{2}} \arrow[r, "j_{H_{2}}"']                          & \mathbb{P}^{1}                  
\end{tikzcd}
\end{center}

A natural way to construct a model for $X_{H_1} \times X_{H_2}$ is to equate the $j$-maps  $j_{H_1}$ and $j_{H_2}$ (see Example \ref{Ex:X03XS45}). For the groups $H_1$ and $H_2$, we will primarily be interested in maximal closed subgroups of $\GL_2(\Z_\ell)$ with surjective determinant, which can be viewed as subgroups of $\GL_2(\Z/\ell^k\Z)$ for sufficiently large $k$. Table \ref{T:Models} lists all such subgroups for $\ell = 2, 3,$ and $5$, together with the $j$-maps  of their associated modular curves. These $j$-maps were computed in Sutherland--Zywina \cite{MR3671434}. 

\renewcommand{\arraystretch}{1.25}
\begin{table}[h]
\begin{tabular}{|lll|}
\hline
RSZB            & Common Name            & The $j$-map $X_H \to \mathbb{P}^1(\mathbb{Q})$                                                                                \\ \hline
\texttt{2.2.0.1}  & $X_{\rm{ns}}(2)$ & $t^2+1728$                                                                           \\
\texttt{2.3.0.1}  & $X_0(2)$               & $\frac{(256-t)^3}{t^2}$                                                              \\
\texttt{4.2.0.1}  & ---                    & $-t^2 + 1728$                                                                        \\
\texttt{4.4.0.1}  & $X_{\rm{ns}}^+(4)$ & $4t^3(8-t)$                                                                          \\
\texttt{8.2.0.1}  & ---                    & $-2t^2+1728$                                                                         \\
\texttt{8.2.0.2}  & ---                    & $2t^2+1728$                                                                          \\
\texttt{3.3.0.1}  & $X_{\rm{ns}}^+(3)$ & $t^3$                                                                                \\
\texttt{3.4.0.1}  & $X_0(3)$               & $\frac{(t+3)^3(t+27)}{t}$                                                            \\
\texttt{9.27.0.1} & ---                    & $\frac{3^7(t^2-1)^3(t^6+3t^5+6t^4+t^3-3t^2+12t+16)^3(2t^3+3t^2-3t-5)}{(t^3-3t-1)^9}$ \\
\texttt{5.5.0.1}  & $X_{S_4}(5)$           & $t^3(t^2+5t+40)$                                                                     \\
\texttt{5.6.0.1}  & $X_0(5)$               & $\frac{(t^2+10t+5)^3}{t}$                                                            \\
\texttt{5.10.0.1} & $X_{\rm{ns}}^+(5)$ & $\frac{8000t^3(t+1)(t^2-5t+10)^3}{(t^2-5)^5}$                                        \\ \hline
\end{tabular}
\caption{All maximal closed subgroups $H$ of $\GL_2(\Z_\ell)$ with surjective determinant for $\ell = 2,3,$ and $5$, along with the $j$-maps of their associated modular curves.}  \label{T:Models} \vspace{-1.5em}
\end{table} 

We now show how these $j$-maps can be used to obtain a model for a fiber product.

\begin{example} \label{Ex:X03XS45}
    The modular curves $X_0(3)$ and $X_{S_4}(5)$ have genus $0$, with $j$-maps given by 
    \[
    j_1(t) = \frac{(t+3)^3(t+27)}{t}
    \quad \text{and} \quad
    j_2(t) = t^3(t^2+5t+40),
    \]
    respectively. By equating the numerator of $j_1(x) - j_2(y)$ with $0$, we obtain an affine plane model
    \[
            x^4 + 36x^3 + 270x^2 - xy^5 - 5xy^4 - 40xy^3 + 756x + 729 = 0
    \]
    for the fiber product $X_0(3) \times X_{S_4}(5)$.
\end{example}

The method illustrated by this example produces a plane model $X_{H_1} \times X_{H_2}$ that typically is singular. In fact, it is known that there are only finitely many modular curves (up to isomorphism over $\overline{\mathbb{Q}}$) that admit a nonsingular plane model \cite{MR4563688}. 
If $X$ is a singular curve and $P \in X(\Q)$ is a singular point, we can consider a desingularization $\tilde{X} \to X$. The preimages of $P$ on $\tilde{X}$ are in bijection with places of the function field $\Q(X)$. There can be several preimages of $P$ on $\tilde{X}$, sometimes with different fields of definition. Nonetheless, if $X$ and $Y$ are two curves defined over $\Q$ that are birational, then there is a bijection between the set of rational places on $X$ and the set of rational places on $Y$. This is because
the desingularization $\tilde{X}$ can be constructed directly from the function field of $X$ (see \cite[Subsection IV.1.1]{MR918564}), together with the fact that two curves are birational if and only if they have isomorphic function fields.

In a few cases, we will require a better model than the one provided by the method illustrated in Example \ref{Ex:X03XS45}. Specifically, for the modular curve $X_0(10) = X_0(2) \times X_0(5)$, we will instead use the  model from \cite{MR3084348}. For the curves $\texttt{9.27.0.1} \times X_{S_4}(5)$ and $\texttt{9.27.0.1} \times X_{0}(5)$, we will instead use Zywina's \texttt{FindModelOfXG} \cite{GitHubZywina} function in Magma. However, for all other fiber products we encounter in this paper, the method of Example \ref{Ex:X03XS45} will suffice.

We conclude our preliminaries with a lemma on nonsurjective primes under isomorphism. Recall the standard fact that the $\overline{\Q}$-isomorphism class of an elliptic curve is determined by its $j$-invariant \cite[Section III.1, Proposition 1.4]{MR2514094}. It is important to note that the surjectivity of the $\ell$-adic Galois representation of an elliptic curve over $\Q$ depends only on its $\overline{\Q}$-isomorphism class. While this fact is well-known to experts in the area, we were unable to locate a reference that proves this exactly, so we include a proof below for completeness (cf.\ \cite[Lemma 3.1]{MR4549128}).
\begin{lemma} \label{L:SurjQBar}
Let $E/\Q$ and $E' / \mathbb{Q}$ be non-CM elliptic curves that are $\overline{\Q}$-isomorphic. Then, for all prime numbers $\ell$, the $\ell$-adic Galois representation $\rho_{E,\ell^\infty}$ is surjective if and only if $\rho_{E',\ell^\infty}$ is surjective.
\end{lemma}
\begin{proof}
    To prove this lemma, it suffices to show that if $\rho_{E,\ell^\infty}$ is surjective, then $\rho_{E',\ell^\infty}$ is also surjective. Since $E$ and $E'$ are non-CM elliptic curves that are isomorphic over $\overline{\Q}$, they must be quadratic twists by \cite[Sec.\ X.5]{MR2514094}. Therefore, for each integer $k\geq 1$, we have
    \( \rho_{E',\ell^k}=\chi_{\ell^k}\cdot \rho_{E,\ell^k}\)
    for some quadratic character $\chi_{\ell^k}\colon \Gal(\overline{\Q}/\Q)\rightarrow \lbrace \pm 1\rbrace$. Thus, the  surjectivity of $\rho_{E,\ell^k}$ implies that
    \begin{equation} \label{E:pm}
    \pm \im \rho_{E',\ell^k} = \GL_2(\Z/\ell^k\Z).
    \end{equation}
    
    Assuming that $\ell \geq 5$, the proof of \cite[Lemma 3.1]{MR4549128} implies that  $\rho_{E',\ell}$ is surjective, and hence   $\rho_{E',\ell^\infty}$ is also surjective by \cite[IV Sec.\ 3.4]{MR1484415}. In the case that $\ell = 2$, by \cite[Corollary 2.13.(i)]{MR2778661}, we know that $\rho_{E',2^\infty}$ is surjective if and only if $\rho_{E',8}$ is surjective. By \eqref{E:pm}, the index of $\im \rho_{E',8}$ in $\GL_2(\Z/8\Z)$ divides 2. However, a quick Magma computation reveals that all index $2$ subgroups of $\GL_2(\Z/8\Z)$ with surjective determinant contain $-I$, which forces $\rho_{E',8}$ to be surjective by \eqref{E:pm}, and thus $\rho_{E',2^\infty}$ is surjective. The case where $\ell = 3$ follows in an almost identical way, using \cite[Corollary 2.13.(ii)]{MR2778661}.
\end{proof}

\section{Proof of Theorem \ref{T:main}} \label{S:Approach}

In this section, we prove Theorem \ref{T:main}. The broad aim is to establish that there are only finitely many $\overline{\Q}$-isomorphism classes of non-CM elliptic curves over $\mathbb{Q}$ whose $2$-adic, $3$-adic, and $5$-adic Galois representations are simultaneously nonsurjective, and to determine the finite set of $j$-invariants corresponding to these $\overline{\Q}$-isomorphism classes. It is then straightforward to check that none of these $j$-invariants come from an elliptic curve whose smallest surjective prime exceeds $7$. As we will see, the proof boils down to numerous rational point computations, which we describe below. These computations were carried out in Magma, using the code in our GitHub repository \cite{GMR-GitHub}.  

To begin, let $E/\mathbb{Q}$ be an elliptic curve without complex multiplication, and assume that
\begin{equation} \label{E:Assumption}
\im \rho_{E,2^\infty} \subsetneq \GL_2(\Z_2), \quad \im \rho_{E,3^\infty} \subsetneq \GL_2(\Z_3), \text{ and} \quad \im \rho_{E,5^\infty} \subsetneq \GL_2(\Z_5). 
\end{equation}
Then $E$ corresponds to a non-cuspidal, non-CM rational point on the fiber product
\begin{equation} \label{E:FiberProd}
    X_{H_2} \times X_{H_3} \times X_{H_5}
\end{equation}
where, for each $\ell \in \{2,3,5\}$, the group $H_\ell$ is a subgroup of $\GL_2(\Z/\ell^k\Z)$ corresponding to a maximal subgroup of $\GL_2(\Z_\ell)$ appearing in Table \ref{T:Models}. Since there are six possibilities for $H_2$, three for $H_3$, and three for $H_5$, there are $54$ modular curves of the form \eqref{E:FiberProd} we may need to  consider. However, we can get away with studying far fewer than this, since it is often the case that the fiber product of just two modular curves among $X_{H_2}, X_{H_3},$ and $X_{H_5}$ already has only finitely many rational points. 

Based on some experimentation, it makes sense to begin by considering the $3$-adic image.  There are three maximal closed subgroups of $\GL_2(\Z_3)$ with surjective determinant, associated with the modular curves $X_0(3)$, \texttt{9.27.0.1}, and $X_{\rm{ns}}^+(3)$. Because each of these curves has infinitely many rational points, we must bring in more information from our assumption \eqref{E:Assumption}. To begin, let us first analyze the case that the $3$-adic image arises from a rational point on $X_0(3)$. For this case, we elect to look at fiber products with the modular curves associated with maximal closed subgroups of $\GL_2(\Z_5)$ with surjective determinant: $X_{S_4}(5)$, $X_{0}(5)$, and $X_{{\rm ns}}^{+}(5)$. Consider the tree displayed below, which summarizes our approach in this first case of $X_0(3)$. 

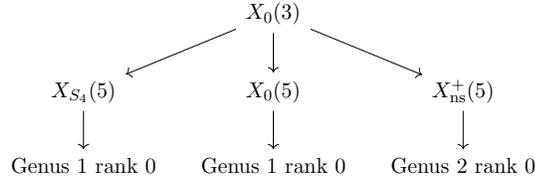
\adjustbox{scale=0.75,center}{
\begin{tikzcd}[column sep=small]
                   & X_{0}(3) \arrow[d] \arrow[ld] \arrow[rd] &                               \\
X_{S_4}(5) \arrow[d] & X_{0}(5) \arrow[d]            & X_{{\rm ns}}^{+}(5) \arrow[d] \\
\text{Genus $1$ rank $0$}   & \text{Genus $1$ rank $0$}                    & \text{Genus $2$ rank 0}        
\end{tikzcd}} 

\noindent Each path down the tree represents a fiber product of modular curves, and we must provably determine the rational points on each of them. The key idea is that if \eqref{E:Assumption} holds and the $3$-adic image arises from $X_0(3)$, then  
$E$ must correspond to a rational point on one of these three fiber products. The text description at the end of each path  briefly notes how this will be accomplished. We now expand on these notes, giving some details about each of the rational point computations.
\begin{itemize}
    \item $X_0(3) \times X_{S_4}(5)$: In Example \ref{Ex:X03XS45}, we obtained the affine model 
    \[
    x^4 + 36x^3 + 270x^2 - xy^5 - 5xy^4 - 40xy^3 + 756x + 729 = 0.
    \]
    Let $X$ be the projective closure of this curve. Using Magma, we find that $X$ is birational to the elliptic curve $E$ defined by the Weierstrass equation
    \[
    y^2 + y = x^3 + x^2 + 2x + 4.
    \]
    This elliptic curve has rank $0$, so all its rational points are torsion, which can be effectively computed. Doing so, we find that
    \[
    E(\mathbb{Q}) = \{ \mathcal{O}, (-1,-2), (-1,1), (2,-5), (2,4) \},
    \]
    where $\mathcal{O}$ denotes the point at infinity. Pulling these back, we determine that
    \[
    X(\mathbb{Q}) = \{ (-81 : -13 : 1), (-27 : 0 : 1), (-9 : 2 : 1), (-3 : 0 : 1), (1 : 0 : 0), (0 : 1 : 0) \}.
    \]
    All of these points are nonsingular except $(-3 : 0 : 1)$ and $(1 : 0 : 0)$. There is only one place above the point $(-3 : 0 : 1)$, and it is not rational since its field of definition has degree $3$. There is a unique place above the point $(1 : 0 : 0)$, and it is rational. Therefore, $X$ has exactly $5$ rational places (which we already knew from the fact that $|E(\mathbb{Q})| = 5$). We list the rational places of $X$, along with an analysis of each, in Table \ref{T:C11} below.

    \begin{table}[h]
\begin{tabular}{|llll|}
\hline
Rational place & $j$-invariant & CM  & Nonsurjective primes \\ \hline
$(-81 : -13 : 1)$   & $-2^4\cdot 3^2\cdot 13^3$          & No  & $2,3,5$              \\
$(-27 : 0 : 1)$     & $0$                & Yes & ---                  \\
$(-9 : 2 : 1)$      & $2^4\cdot 3^3$              & No  & $2,3,5$              \\
$(1 : 0 : 0)$       & $\infty$           & --- & ---                  \\
$(0 : 1 : 0)$       & $\infty$           & --- & ---                  \\ \hline
\end{tabular}
\caption{Analysis of rational places of $X_0(3) \times X_{S_4}(5)$} \label{T:C11} \vspace{-1.5em}
\end{table}

\noindent To elaborate on the analysis, let us write $E_j$ to denote any elliptic curve with $j$-invariant $j$. We note that $E_0$ is CM, whereas $E_{-2^4\cdot 3^2\cdot 13^3}$ and $E_{2^4\cdot 3^3}$ are non-CM. For both of these curves, we use \cite{RSZB-GitHub} to compute that the $\ell$-adic Galois representation is nonsurjective for $\ell \in \{2,3,5\}$ and surjective for all primes $\ell \geq 7$. Finally, Lemma \ref{L:SurjQBar} implies that whether or not $\ell$ is surjective is invariant on $\overline{\Q}$-isomorphism classes.

    \item $X_0(3) \times X_0(5) = X_0(15)$: The rational points on $X_0(15)$ are well-known \cite{MR376533}, but we recompute them here for completeness using the same method as in the first case. Using the method of Example \ref{Ex:X03XS45}, we obtain the affine model
    \[
    x^4y + 36x^3y + 270x^2y - xy^6 - 30xy^5 - 315xy^4 - 1300xy^3 - 1575xy^2 + 6xy - 125x + 729y = 0.
    \]
    Let $X$ be the projective closure of this curve. We find that $X$ is birational to an elliptic curve $E$ of rank $0$. Pulling back the rational points of $E$, we determine that 
    \[
    X(\mathbb{Q}) = \{(-\tfrac{729}{2} : -40 : 1), (-32 : -\tfrac{25}{2} : 1), (-\tfrac{729}{32} : -\tfrac{25}{8} : 1), (-2 : -10 : 1), (1 : 0 : 0), (0:1:0), (0 : 0 : 1) \}.
    \]
    All of these points are nonsingular except $(1 : 0 : 0)$. There are two places above the point $(1 : 0 : 0)$, both of which are rational. We denote these by $(1 : 0 : 0)_1$ and $(1 : 0 : 0)_2$. We list the rational places of $X$, along with an analysis of each, in Table \ref{T:C12} below.
    
    \begin{table}[h]
    \begin{tabular}{|llll|} \hline
    Rational place & $j$-invariant & CM  & Nonsurjective primes \\ \hline
    $(-\frac{729}{2} : -40 : 1)$            & $-2^{-3}\cdot 5^2\cdot 241^3$ & No  & $2,3,5$   \\
    $(-32 : -\frac{25}{2} : 1)$             & $-2^{-5}\cdot 5\cdot 29^3$  & No  & $2,3,5$     \\
    $(-\frac{729}{32} : -\frac{25}{8} : 1)$ & $2^{-15}\cdot 5 \cdot 211^3$ & No  & $2,3,5$  \\
    $(-2 : -10 : 1)$                        & $-2^{-1}\cdot 5^2$         & No  & $2,3,5$  \\
    $(1 : 0 : 0)_1$                           & $\infty$                  & ---  & --- \\
    $(1 : 0 : 0)_2$                           & $\infty$                  & ---  & --- \\
    $(0 : 1 : 0)$                           & $\infty$                 & ---  & ---  \\
    $(0 : 0 : 1)$                           & $\infty$   & ---  & --- \\ \hline
    \end{tabular}
        \caption{Analysis of rational places of $X_0(3) \times X_0(5)$} \label{T:C12} \vspace{-1.5em}
    \end{table}
    
    \noindent Each of the non-cuspidal $j$-invariants correspond to non-CM elliptic curves for which the $\ell$-adic Galois representation is nonsurjective for $\ell \in \{2,3,5\}$ and surjective for all $\ell \geq 7$.
        
    \item $X_0(3) \times  X_{{\rm ns}}^{+}(5)$: Using the method of Example \ref{Ex:X03XS45}, we obtain the affine model
    \begin{align*}
    0 = x^4&y^{10} - 25x^4y^8 + 250x^4y^6 - 1250x^4y^4 + 3125x^4y^2 - 3125x^4 + 36x^3y^{10} - 900x^3y^8 \\ 
    &+ 9000x^3y^6 - 45000x^3y^4 + 112500x^3y^2 - 112500x^3 + 270x^2y^{10} - 6750x^2y^8 \\
    &+ 67500x^2y^6 - 337500x^2y^4 + 843750x^2y^2 - 843750x^2 - 7244xy^{10} + 112000xy^9 \\
    & - 738900xy^8 + 2560000xy^7 - 4811000xy^6 + 3600000xy^5 + 3055000xy^4 \\
    &- 8000000xy^3 + 2362500xy^2 - 2362500x + 729y^{10} - 18225y^8 + 182250y^6 \\
    &- 911250y^4 + 2278125y^2 - 2278125.
    \end{align*}
    Let $X$ be the projective closure of this curve. We find that $X$ is birational to  the hyperelliptic curve $H$ defined by the Weierstrass equation
    \[
    y^2 = 9x^6 - 6x^5 - 35x^4 + 40x^2 + 12x - 8.
    \]
    Since $H$ is a genus $2$ curve whose  Jacobian has rank $0$,  we can use Magma's \texttt{Chabauty0} command to compute its rational points. Performing this computation and pulling back to $X$, we obtain the complete set of rational points on $X$. Three of these points are singular, and none of the singular points have a rational place above them. Therefore, we have determined all rational places of $X$, which we list, along with their analysis, in Table \ref{T:C13}.
    
    \begin{table}[h]
    \begin{tabular}{|llll|} \hline
    Rational place & $j$-invariant & CM  & Nonsurjective primes \\ \hline
    $(-243 : 2 : 1)$ & $-2^{15}\cdot 3\cdot 5^3$ & Yes & --- \\
    $(-27 : -1 : 1)$ & $0$      & Yes & ---    \\
    $(-27 : 0 : 1)$  & $0$      & Yes & ---    \\
    $(-3 : -1 : 1)$  & $0$      & Yes & ---    \\
    $(27 : 3 : 1)$   & $2^4\cdot 3^3\cdot 5^3$  & Yes & ---   \\ \hline
    \end{tabular}
    \caption{Analysis of rational places of $X_0(3) \times X_{{\rm ns}}^{+}(5)$} \label{T:C13} \vspace{-1.5em}
    \end{table}
        
    \noindent  Notice that each rational place of $X$ corresponds to a CM $j$-invariant.
    
\end{itemize}

For the second case, we assume that the $3$-adic image arises from a rational point on \texttt{9.27.0.1}. For this case, we elect to look at fiber products with the modular curves associated with maximal closed subgroups of $\GL_2(\Z_2)$ with surjective determinant. One of the fiber products that arises is  $\texttt{9.27.0.1} \times X_{\rm{ns}}^+(4)$ which has only finitely many rational points, but it appears somewhat challenging to provably determine them. Thus, we approach this subcase by further considering 5-adic images. The tree below shows all the fiber products we will consider, along with a brief description of the approach for each.

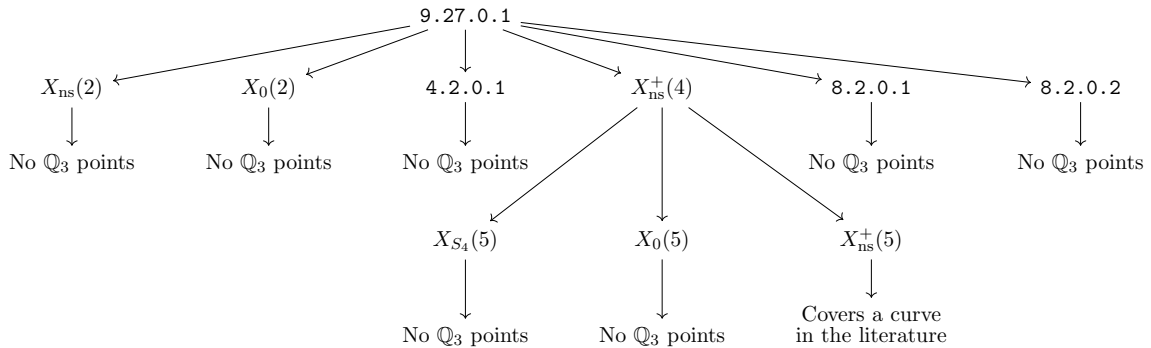
\adjustbox{scale=0.75,center}{
\begin{tikzcd}
    && \texttt{9.27.0.1} &&&  \\
    {X_{\rm{ns}}(2)} & {X_{0}(2)} & {\texttt{4.2.0.1}} & {X_{\rm{ns}}^+(4)} & \texttt{8.2.0.1} & {\texttt{8.2.0.2}} \\
    {\text{No } \mathbb{Q}_{3} \text{ points}} & {\text{No } \mathbb{Q}_{3} \text{ points}} & {\text{No } \mathbb{Q}_{3} \text{ points}} && {\text{No } \mathbb{Q}_{3} \text{ points}} & {\text{No } \mathbb{Q}_{3} \text{ points}} \\
    && {X_{S_4}(5)} & {X_0(5)} & {X_{\rm{ns}}^+(5)} \\
    && {\text{No } \mathbb{Q}_{3} \text{ points}} & {\text{No } \mathbb{Q}_{3} \text{ points}} & \shortstack{\text{Covers a curve} \\ \text{in the literature}}
    \arrow[from=1-3, to=2-1]
    \arrow[from=1-3, to=2-2]
    \arrow[from=1-3, to=2-3]
    \arrow[from=1-3, to=2-4]
    \arrow[from=1-3, to=2-5]
    \arrow[from=1-3, to=2-6]
    \arrow[from=2-1, to=3-1]
    \arrow[from=2-2, to=3-2]
    \arrow[from=2-3, to=3-3]
    \arrow[from=2-4, to=4-3]
    \arrow[from=2-4, to=4-4]
    \arrow[from=2-4, to=4-5]
    \arrow[from=2-5, to=3-5]
    \arrow[from=2-6, to=3-6]
    \arrow[from=4-3, to=5-3]
    \arrow[from=4-4, to=5-4]
    \arrow[from=4-5, to=5-5]
\end{tikzcd}} \vspace{.25em}

We now discuss the rational point computations for each fiber product in the tree above.

\begin{itemize}
    \item $\texttt{9.27.0.1} \times X_{\rm{ns}}(2)$: We construct a model of $X = \texttt{9.27.0.1} \times X_{\rm{ns}}(2)$ as a projective plane curve using the method of Example \ref{Ex:X03XS45}. We omit the equation here as it is rather lengthy. Using Magma, we compute a birational map $\varphi \colon X \to H$ to the hyperelliptic curve $H$ defined by the Weierstrass equation
    \[
     y^2 = 6x^6 - 9x^5 - 18x^4 + 33x^3 + 9x^2 - 36x - 12.
    \]
    Using the command \texttt{IsLocallySolvable}, we find that $H(\Q_3) = \emptyset$, and hence $H(\Q) = \emptyset$. Therefore, $X$ has no rational places. We can check this a second way by computing the base locus of $\varphi$. Performing this computation, we find that $X(\mathbb{Q}) = \{ (0:1:0), (1:0:0) \}$. The points $(0:1:0)$ and $(1:0:0)$ are singular and each has only a single place above it, of degree $2$ and $3$, respectively.

    \item $\texttt{9.27.0.1} \times X_{0}(2)$: We construct a model of $X = \texttt{9.27.0.1} \times X_{0}(2)$ as a projective plane curve using the method of Example \ref{Ex:X03XS45}. In this case, $X$ has genus $4$ and is not hyperelliptic. Using Magma, we compute the canonical model $H$ of $X$. By calling the command  \texttt{IsLocallySolvable}  on $H$, we find that $H(\Q_3) = \emptyset$, and hence $H(\Q)=\emptyset$. Therefore, $X$ has no rational places.
        
    \item $\texttt{9.27.0.1} \times \texttt{4.2.0.1}$: This curve is handled in the same way as  $\texttt{9.27.0.1} \times X_{\rm{ns}}(2)$.
        
    \item $\texttt{9.27.0.1} \times X_{\rm{ns}}^+(4)$: This curve has genus $6$ and appears in \cite[p.\ 11]{DG}, although the authors did not determine its rational points. In this paper, rather than attempting to compute its rational points, we instead  refine our consideration by looking at fiber products with the modular curves corresponding to the three maximal closed subgroups of $\GL_2(\Z_5)$ with surjective determinant:
    \begin{itemize}
        \item  $\texttt{9.27.0.1} \times X_{\rm{ns}}^+(4) \times X_{S_4}(5)$: This curve covers the modular curve $X = \texttt{9.27.0.1} \times X_{S_4}(5)$. We use Zywina's \texttt{FindModelOfXG} to construct a nonsingular model of $X$. Running Zywina's \texttt{PointsViaLifting}, we find that $X(\Q_3) = \emptyset$, and hence $X(\Q) = \emptyset$.

        \item  $\texttt{9.27.0.1} \times X_{\rm{ns}}^+(4) \times X_{0}(5)$: This curve covers the modular curve $\texttt{9.27.0.1} \times X_{0}(5)$, which is handled in the same way as  $\texttt{9.27.0.1} \times X_{S_4}(5)$.
        \item  $\texttt{9.27.0.1} \times X_{\rm{ns}}^+(4) \times X_{\rm{ns}}^+(5)$: This curve covers the modular curve $X_{\rm{ns}}^+(4) \times X_{\rm{ns}}^+(5)$, all of whose rational points are CM by \cite[p.\ 2760]{MR2684496} and \cite{DG}.
    \end{itemize}
    \item $\texttt{9.27.0.1} \times \texttt{8.2.0.1}$: This curve is handled in the same way as  $\texttt{9.27.0.1} \times X_{\rm{ns}}(2)$.
    \item $\texttt{9.27.0.1} \times \texttt{8.2.0.2}$: This curve is handled in the same way as  $\texttt{9.27.0.1} \times X_{\rm{ns}}(2)$.
\end{itemize}

The third case is a bit more  intricate. Here, we assume that the $3$-adic image arises from a rational point on $X_{\rm{ns}}^+(3)$. We elect to look at fiber products with the modular curves associated with maximal
closed subgroups of $\GL_2(\Z_2)$ with surjective determinant. All of these fiber products have finitely many rational points, except $X_{{\rm ns}}^{+}(3) \times X_0(2)$ and $X_{{\rm ns}}^{+}(3) \times X_{\rm{ns}}^+(4)$. In the tree diagram below, we indicate that these two fiber products are isomorphic to $\mathbb{P}^1$. For these cases, we further refine our search by taking fiber products with the modular curves associated with maximal closed subgroups of $\GL_2(\Z_5)$ with surjective determinant.

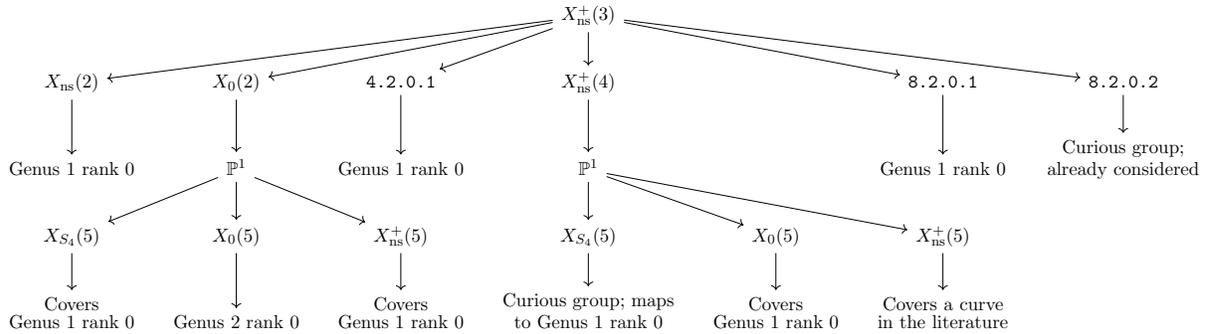
\adjustbox{scale=0.65,center}{
\begin{tikzcd}[column sep=small]
                       &                                                &                               & X_{{\rm ns}}^{+}(3) \arrow[llld] \arrow[ld] \arrow[lld] \arrow[d] \arrow[rrd] \arrow[rrrd] &                    &                               &                                                   &                               &                               \\
X_{\rm{ns}}(2) \arrow[d]    & X_{0}(2) \arrow[d]                             & \texttt{4.2.0.1} \arrow[d]           & X_{\rm{ns}}^+(4) \arrow[d]                                                                        &                    & \texttt{8.2.0.1} \arrow[d]           & \texttt{8.2.0.2} \arrow[d]                               &                               &                               \\
\text{Genus $1$ rank $0$}       & \mathbb{P}^{1} \arrow[ld] \arrow[d] \arrow[rd] & \text{Genus $1$ rank $0$}              & \mathbb{P}^{1} \arrow[d] \arrow[rd] \arrow[rrd]                                            &                    & \text{Genus $1$ rank $0$}              & \shortstack{\text{Curious group;} \\ \text{already considered}} &                               &                               \\
X_{S_{4}}(5) \arrow[d] & X_{0}(5) \arrow[d]                             & X_{{\rm ns}}^{+}(5) \arrow[d] & X_{S_{4}}(5) \arrow[d]                                                                     & X_{0}(5) \arrow[d] & X_{{\rm ns}}^{+}(5) \arrow[d] &                                 &  &  \\
\shortstack{\text{Covers} \\ \text{Genus $1$ rank $0$}}  & \text{Genus $2$ rank $0$}                          & \shortstack{\text{Covers} \\ \text{Genus $1$ rank $0$}}         & \shortstack{\text{Curious group; maps} \\ \text{to Genus $1$ rank $0$}}                                                             & \shortstack{\text{Covers} \\ \text{Genus $1$ rank $0$}}   & \shortstack{\text{Covers a curve} \\ \text{in the literature}}      &  &       &  \\
\end{tikzcd}}

We now discuss the rational point computations for each fiber product in the tree above.
\begin{itemize}
    \item $X_{{\rm ns}}^{+}(3) \times X_{\rm{ns}}(2)$: This is an elliptic curve of rank $0$. A similar analysis to that for $X_0(3) \times X_{S_4}(5)$ reveals that every rational point on this modular curve is cuspidal or CM.
    \item $X_{{\rm ns}}^{+}(3) \times X_{0}(2)$: This curve has genus $0$, and it has infinitely many points. As mentioned above, we further refine our search by taking fiber products with the modular curves associated with maximal closed subgroups of $\GL_2(\Z_5)$ with surjective determinant.
    \begin{itemize}
        \item $X_{{\rm ns}}^{+}(3) \times X_{0}(2) \times X_{S_4}(5)$: This curve covers the modular curve $X_0(2) \times X_{S_4}(5)$, which is an elliptic curve of rank $0$. A similar analysis to that for $X_0(3) \times X_{S_4}(5)$ reveals that every rational point on this modular curve is cuspidal or CM.
        \item $X_{{\rm ns}}^{+}(3) \times X_{0}(2) \times X_0(5)$: This curve has genus $2$ and its Jacobian has rank $0$. We construct a projective model for it as $X_{{\rm ns}}^{+}(3) \times X_{0}(10)$ using the method of Example \ref{Ex:X03XS45}. From here, a similar analysis to that for $X_0(3) \times X_{\rm ns}^+(5)$  reveals that
every rational point on this modular curve is cuspidal.
        \item $X_{{\rm ns}}^{+}(3) \times X_{0}(2) \times X_{\rm{ns}}^+(5)$: This curve covers the modular curve $X_{0}(2) \times X_{\rm{ns}}^+(5)$, which is an elliptic curve of rank $0$. A similar analysis to that for $X_0(3) \times X_{S_4}(5)$ reveals that every rational point on this modular curve is CM.
    \end{itemize}
    \item $X_{{\rm ns}}^{+}(3) \times \texttt{4.2.0.1}$: This is an elliptic curve of rank $0$. A similar analysis to that for $X_0(3) \times X_{S_4}(5)$ reveals that every rational point on this modular curve is cuspidal or CM.
    \item $X_{{\rm ns}}^{+}(3) \times X_{\rm{ns}}^+(4)$: This curve has genus $0$, and it has infinitely many points. Again, we further refine our search by taking fiber products with the modular curves associated with maximal closed subgroups of $\GL_2(\Z_5)$ with surjective determinant.
    \begin{itemize}
        \item $X_{{\rm ns}}^{+}(3) \times X_{\rm{ns}}^+(4) \times X_{S_4}(5)$: This curve covers $X_{{\rm ns}}^{+}(3) \times X_{S_4}(5)$. In \cite{chiloyan2023classificationcuriousgaloisgroups}, the author showed that this curve has the ``curious'' property that  every rational point on it arises from a rational point on $X_{{\rm ns}}^{+}(3) \times X_{\rm{sp}}^+(5)$. Thus, every rational point on $X_{{\rm ns}}^{+}(3) \times X_{\rm{ns}}^+(4) \times X_{S_4}(5)$ arises from a rational point on $X_{{\rm ns}}^{+}(3) \times X_{\rm{ns}}^+(4) \times X_{\rm{sp}}^+(5)$, which covers $X = X_{\rm{ns}}^+(4) \times X_{\rm{sp}}^+(5)$. Using the method of Example \ref{Ex:X03XS45}, we compute a projective model for $X$. We find that $X$ is a plane quartic of genus $3$. We compute the canonical image of $X$ and apply the command \texttt{MinimizeReducePlaneQuartic} to obtain a simpler curve $H$, which is birational to $X$. Using the command \texttt{AutomorphismGroupOfPlaneQuartic}, we find a nontrivial automorphism $\varphi$ of $H$. We quotient $H$ by $\varphi$ using the command \texttt{AutomorphismGroup}. The quotient curve is an elliptic curve of rank $0$. By computing its rational points and pulling them back to $X$, we obtain all rational points of $X$. Two of these are singular: $(0:-5:1)$ and $(1:0:0)$. The former has no rational places above it and the latter has a unique rational place above it. In Table \ref{T:C3A}, we list the rational places of $X$ and their analysis.
    \begin{table}[h]
        \begin{tabular}{|llll|} \hline
Rational place & $j$-invariant & CM  & Nonsurjective primes \\ \hline
$(-\frac{561}{8} : -\frac{5}{4} : 1)$ & $-2^{-10} \cdot 3^3 \cdot 5^4 \cdot 11^3 \cdot 17^3$ & No  & $ 2, 5$ \\
$(-8 : -3 : 1)$        & $-2^{15}$                   & Yes & ---       \\
$(1 : 0 : 0)$          & $\infty$                  & --- & ---       \\
$(6 : -2 : 1)$         & $2^6 \cdot 3^3$                    & Yes & ---       \\
$(8 : -5 : 1)$         & $0$                       & Yes & ---       \\ 
$(24 : -1 : 1)$        & $-2^{15} \cdot 3^3$                 & Yes & ---       \\ \hline
\end{tabular}
    \caption{Analysis of rational places of $X_{{\rm ns}}^{+}(4) \times X_{\rm sp}^+(5)$} \label{T:C3A} \vspace{-1.5em}
    \end{table}      

    Observe that the only non-cuspidal, non-CM rational place of $X$ is $(-\frac{561}{8} : -\frac{5}{4} : 1)$, which corresponds to the $j$-invariant $-2^{-10} \cdot 3^3 \cdot 5^4 \cdot 11^3 \cdot 17^3$. Any elliptic curve with this $j$-invariant has that $3$ is its smallest surjective prime.
        
        \item $X_{{\rm ns}}^{+}(3) \times X_{\rm{ns}}^+(4) \times X_0(5)$: This curve covers the modular curve $X_{\rm{ns}}^+(4) \times X_0(5)$, which is an elliptic curve of rank $0$. A similar analysis to that for $X_0(3) \times X_{S_4}(5)$ provably determines the rational points of $X$. In Table \ref{T:C3B}, we list the rational places of $X$ and their analysis.

    \begin{table}[h]
        \begin{tabular}{|llll|} \hline
Rational place                  & $j$-invariant      & CM    & Nonsurjective primes \\ \hline
$(1 : 0 : 0)_1$      & $\infty$           & ---   & ---                 \\
$(1 : 0 : 0)_2$      & $\infty$           & ---   & ---                 \\
$(\frac{59}{8} : -\frac{25}{4} : 1)$   & $2^{-10} \cdot 5 \cdot 59^3$ & No    & $2, 5$          \\
$(\frac{41}{2} : -20 : 1)$     & $-2^{-2} \cdot 5^2 \cdot 41^3$   & No    & $2, 5$          \\ \hline
\end{tabular}
           \caption{Analysis of rational places of $X_{\rm{ns}}^+(4) \times X_0(5)$} \label{T:C3B} \vspace{-1.5em}
    \end{table}    

     Observe that the rational points $(\frac{59}{8} : -\frac{25}{4} : 1)$ and $(\frac{41}{2} : -20 : 1)$ are non-cuspidal and non-CM, but the smallest surjective prime associated with elliptic curves arising from these points is $3$.

        \item $X_{{\rm ns}}^{+}(3) \times X_{\rm{ns}}^+(4) \times X_{\rm{ns}}^+(5)$: This curve covers the modular curve $X_{\rm{ns}}^+(4) \times X_{\rm{ns}}^+(5)$, all of whose rational points are CM by \cite[p.\ 2760]{MR2684496} and \cite{DG}. 
    \end{itemize}
    \item $X_{{\rm ns}}^{+}(3) \times \texttt{8.2.0.1}$: This is an elliptic curve of rank $0$. A similar analysis to that for $X_0(3) \times X_{S_4}(5)$ reveals that every rational point on this modular curve is cuspidal or CM.
    \item $X_{{\rm ns}}^{+}(3) \times \texttt{8.2.0.2}$: In \cite{MR4588927}, the authors show that this curve has the ``curious'' property that every rational point on it arises from a rational point on $X_{{\rm ns}}^{+}(3) \times X_0(2)$, which we have already considered.
\end{itemize}

Let $\mathcal{J}$ denote the set of six $j$-invariants appearing in the statement of Theorem \ref{T:main}. The computations above prove that any non-CM elliptic curve $E/\mathbb{Q}$ whose $2$-, $3$-, and $5$-adic Galois representations are all nonsurjective must have $j$-invariant belonging to $\mathcal{J}$. We have seen that the smallest surjective prime associated with any elliptic curve whose $j$-invariant belongs to $\mathcal{J}$ is $7$. Consequently, the smallest surjective prime for any non-CM elliptic curve over $\mathbb{Q}$ is at most $7$. Moreover, any non-CM elliptic curve over $\mathbb{Q}$ whose $j$-invariant is not in $\mathcal{J}$ has smallest surjective prime at most $5$.

\bibliographystyle{amsplain}
\bibliography{References}

\providecommand{\bysame}{\leavevmode\hbox to3em{\hrulefill}\thinspace}
\providecommand{\MR}{\relax\ifhmode\unskip\space\fi MR }
\providecommand{\MRhref}[2]{%
  \href{http://www.ams.org/mathscinet-getitem?mr=#1}{#2}
}
\providecommand{\href}[2]{#2}
\begin{thebibliography}{10}

\bibitem{MR4563688}
Samuele Anni, Eran Assaf, and Elisa Lorenzo~Garc\'{\i}a, \emph{On smooth plane models for modular curves of {S}himura type}, Res. Number Theory \textbf{9} (2023), no.~2, Paper No. 21, 20. \MR{4563688}

\bibitem{MR4732686}
Barinder~S. Banwait, Armand Brumer, Hyun~Jong Kim, Zev Klagsbrun, Jacob Mayle, Padmavathi Srinivasan, and Isabel Vogt, \emph{Computing nonsurjective primes associated to {G}alois representations of genus {$2$} curves}, Lu{C}a{NT}: {LMFDB}, computation, and number theory, Contemp. Math., vol. 796, Amer. Math. Soc., [Providence], RI, [2024] \copyright 2024, pp.~129--163. \MR{4732686}

\bibitem{MR2684496}
Burcu Baran, \emph{Normalizers of non-split {C}artan subgroups, modular curves, and the class number one problem}, J. Number Theory \textbf{130} (2010), no.~12, 2753--2772. \MR{2684496}

\bibitem{MR2753610}
Yuri Bilu and Pierre Parent, \emph{Serre's uniformity problem in the split {C}artan case}, Ann. of Math. (2) \textbf{173} (2011), no.~1, 569--584. \MR{2753610}

\bibitem{MR376533}
B.~J. Birch and W.~Kuyk (eds.), \emph{Modular functions of one variable. {IV}}, Lecture Notes in Mathematics, Vol. 476, Springer-Verlag, Berlin-New York, 1975. \MR{376533}

\bibitem{MAGMA}
Wieb Bosma, John Cannon, and Catherine Playoust, \emph{The {M}agma algebra system. {I}. {T}he user language}, J. Symbolic Comput. \textbf{24} (1997), no.~3-4, 235--265, Computational algebra and number theory (London, 1993).

\bibitem{BELOV}
Abbey Bourdon, \"Ozlem Ejder, Yuan Liu, Frances Odumodu, and Bianca Viray, \emph{On the level of modular curves that give rise to isolated {$j$}-invariants}, Adv. Math. \textbf{357} (2019), 106824, 33. \MR{4016915}

\bibitem{chiloyan2023classificationcuriousgaloisgroups}
Garen Chiloyan, \emph{A classification of curious galois groups as direct products}, 2023, arXiv:2310.19987.

\bibitem{DG}
Harris~B. Daniels and Enrique Gonz\'alez-Jim\'enez, \emph{Serre's constant of elliptic curves over the rationals}, Exp. Math. \textbf{31} (2022), no.~2, 518--536. \MR{4458130}

\bibitem{MR4588927}
Harris~B. Daniels and \'{A}lvaro Lozano-Robledo, \emph{Coincidences of division fields}, Ann. Inst. Fourier (Grenoble) \textbf{73} (2023), no.~1, 163--202. \MR{4588927}

\bibitem{MR337993}
P.~Deligne and M.~Rapoport, \emph{Les sch\'{e}mas de modules de courbes elliptiques}, Modular functions of one variable, {II} ({P}roc. {I}nternat. {S}ummer {S}chool, {U}niv. {A}ntwerp, {A}ntwerp, 1972), Lecture Notes in Math., Vol. 349, Springer, Berlin-New York, 1973, pp.~143--316. \MR{337993}

\bibitem{MR1969642}
Luis~V. Dieulefait, \emph{Explicit determination of the images of the {G}alois representations attached to abelian surfaces with {${\rm End}(A)=\mathbb Z$}}, Experiment. Math. \textbf{11} (2002), no.~4, 503--512 (2003). \MR{1969642}

\bibitem{furio2023serresuniformityquestionproper}
Lorenzo Furio and Davide Lombardo, \emph{Serre's uniformity question and proper subgroups of ${C}_{ns}^+(p)$}, 2023, arXiv:2305.17780.

\bibitem{GMR-GitHub}
Tyler Genao, Jacob Mayle, and Jeremy Rouse, \emph{Github repository associated with ``{A} uniform bound on the smallest surjective prime''}, \url{https://github.com/maylejacobj/SmallestSurjectivePrime}, 2025.

\bibitem{MR2778661}
Aaron Greicius, \emph{Elliptic curves with surjective adelic {G}alois representations}, Experiment. Math. \textbf{19} (2010), no.~4, 495--507. \MR{2778661}

\bibitem{MR3161774}
Eric Larson and Dmitry Vaintrob, \emph{On the surjectivity of {G}alois representations associated to elliptic curves over number fields}, Bull. Lond. Math. Soc. \textbf{46} (2014), no.~1, 197--209. \MR{3161774}

\bibitem{LMFDB}
The {LMFDB Collaboration}, \emph{The {L}-functions and modular forms database}, \url{https://www.lmfdb.org}, 2024, [Online; accessed 14 November 2024].

\bibitem{MR3084348}
\'{A}lvaro Lozano-Robledo, \emph{On the field of definition of {$p$}-torsion points on elliptic curves over the rationals}, Math. Ann. \textbf{357} (2013), no.~1, 279--305. \MR{3084348}

\bibitem{MazurSwD}
B.~Mazur and P.~Swinnerton-Dyer, \emph{Arithmetic of {W}eil curves}, Invent. Math. \textbf{25} (1974), 1--61. \MR{354674}

\bibitem{Ma1978}
Barry Mazur, \emph{Rational isogenies of prime degree (with an appendix by {D}. {G}oldfeld)}, Invent. Math. \textbf{44} (1978), no.~2, 129--162. \MR{482230}

\bibitem{MR3957898}
Jackson~S. Morrow, \emph{Composite images of {G}alois for elliptic curves over {$\mathbf{Q}$} and entanglement fields}, Math. Comp. \textbf{88} (2019), no.~319, 2389--2421. \MR{3957898}

\bibitem{RSZB-GitHub}
Jeremy Rouse, Andrew~V. Sutherland, and David Zureick-Brown, \emph{ell-adic-galois-images}, \url{https://github.com/AndrewVSutherland/ell-adic-galois-images}, 2021, GitHub repository.

\bibitem{MR4468989}
Jeremy Rouse, Andrew~V. Sutherland, and David Zureick-Brown, \emph{{$\ell$}-adic images of {G}alois for elliptic curves over {$\mathbb{Q}$} (and an appendix with {J}ohn {V}oight)}, Forum Math. Sigma \textbf{10} (2022), Paper No. e62, 63, With an appendix with John Voight. \MR{4468989}

\bibitem{MR387283}
Jean-Pierre Serre, \emph{Propri\'{e}t\'{e}s galoisiennes des points d'ordre fini des courbes elliptiques}, Invent. Math. \textbf{15} (1972), no.~4, 259--331. \MR{387283}

\bibitem{MR0644559}
\bysame, \emph{Quelques applications du th\'{e}or\`eme de densit\'{e} de {C}hebotarev}, Inst. Hautes \'{E}tudes Sci. Publ. Math. (1981), no.~54, 323--401. \MR{644559}

\bibitem{MR918564}
\bysame, \emph{Algebraic groups and class fields}, Graduate Texts in Mathematics, vol. 117, Springer-Verlag, New York, 1988, Translated from the French. \MR{918564}

\bibitem{MR1484415}
\bysame, \emph{Abelian {$l$}-adic representations and elliptic curves}, Research Notes in Mathematics, vol.~7, A K Peters, Ltd., Wellesley, MA, 1998, With the collaboration of Willem Kuyk and John Labute, Revised reprint of the 1968 original. \MR{1484415}

\bibitem{MR2514094}
Joseph~H. Silverman, \emph{The arithmetic of elliptic curves}, second ed., Graduate Texts in Mathematics, vol. 106, Springer, Dordrecht, 2009. \MR{2514094}

\bibitem{3482279}
Andrew~V. Sutherland, \emph{Computing images of {G}alois representations attached to elliptic curves}, Forum Math. Sigma \textbf{4} (2016), Paper No. e4, 79. \MR{3482279}

\bibitem{MR3671434}
Andrew~V. Sutherland and David Zywina, \emph{Modular curves of prime-power level with infinitely many rational points}, Algebra Number Theory \textbf{11} (2017), no.~5, 1199--1229. \MR{3671434}

\bibitem{Velu}
Jacques V\'elu, \emph{Les points rationnels de {$X\sb{0}(37)$}}, Journ\'ees {A}rithm\'etiques ({G}renoble, 1973), Suppl\'ement au Bull. Soc. Math. France, vol. Tome 102, Soc. Math. France, Paris, 1974, pp.~169--179. \MR{366930}

\bibitem{Zy2015a}
David Zywina, \emph{On the possible images of the mod $\ell$ representations associated to elliptic curves over {$\mathbb{Q}$}}, 2015, arXiv:1508.07660.

\bibitem{zywina2022explicit}
\bysame, \emph{Explicit open images for elliptic curves over $\mathbb{Q}$}, 2022, arXiv:2206.14959.

\bibitem{MR4549128}
\bysame, \emph{On the surjectivity of {${\rm mod}\,\ell$} representations associated to elliptic curves}, Bull. Lond. Math. Soc. \textbf{54} (2022), no.~6, 2404--2417. \MR{4549128}

\bibitem{GitHubZywina}
\bysame, \emph{\emph{GitHub repository related to} explicit open images for elliptic curves over $\mathbb{Q}$}, 2023, \url{https://github.com/davidzywina/OpenImage}.

\end{thebibliography}

\end{document}